\documentclass[10pt,reqno]{amsart}

\usepackage{amssymb,latexsym,mathrsfs,amsmath}
\usepackage{amsthm}
\usepackage{graphicx}
\usepackage{diagbox}
\usepackage{tikz}
\usetikzlibrary{arrows}

\newtheorem{lemma}{Lemma}[section]
\newtheorem{proposition}[lemma]{Proposition}
\newtheorem{theorem}[lemma]{Theorem}
\newtheorem{corollary}[lemma]{Corollary}
\newtheorem{conjecture}[lemma]{Conjecture}

\theoremstyle{definition}
\newtheorem{remark}[lemma]{Remark}
\newtheorem{definition}[lemma]{Definition}
\newtheorem{example}[lemma]{Example}

\DeclareMathOperator{\id}{id}
\DeclareMathOperator{\chr}{char}
\DeclareMathOperator{\gl}{gl}
\DeclareMathOperator{\im}{im}
\DeclareMathOperator{\Kh}{Kh}

\DeclareMathOperator{\Sq}{Sq}
\DeclareMathOperator{\Tor}{Tor}

\newcommand{\Z}{\mathbb{Z}}

\newcommand{\Q}{\mathbb{Q}}

\newcommand{\F}{\mathbb{F}}
\newcommand{\K}{\mathbb{K}}
\newcommand{\BN}{H_{\mathrm{BN}}}
\newcommand{\tBN}{\tilde{H}_{\mathrm{BN}}}

\newcommand{\CBN}{C_{\mathrm{BN}}}
\newcommand{\tCBN}{\tilde{C}_{\mathrm{BN}}}

\newcommand{\CKh}{C_{\Kh}}

\newcommand{\Rone}{\mathbb{Z}}
\newcommand{\Rmor}[1]{\mathbb{Z}^{#1}}
\newcommand{\Tone}[1]{\mathbb{Z}/#1 \mathbb{Z}}
\newcommand{\Tmor}[2]{(\mathbb{Z}/#1 \mathbb{Z})^{#2}}

\begin{document}
\parindent0em
\setlength\parskip{.1cm}
\thispagestyle{empty}
\title{On an integral version of the Rasmussen invariant}
\author[Dirk Sch\"utz]{Dirk Sch\"utz}
\address{Department of Mathematical Sciences\\ Durham University\\ United Kingdom}
\email{dirk.schuetz@durham.ac.uk}

\begin {abstract}
We define a Rasmussen $s$-invariant over the coefficient ring $\Z$, and show how it is related to the $s$-invariants defined over a field. A lower bound for the slice genus of a knot arising from it is obtained, and we give examples of knots for which this lower bound is better than all lower bounds coming from the $s$-invariants over fields. We also compare it to the Lipshitz-Sarkar refinement related to the first Steenrod square.
\end {abstract}

\maketitle

\section{Introduction}
The Rasmussen $s$-invariant \cite{MR2729272} is one of the most important and studied objects arising from Khovanov homology \cite{MR1740682}. Originally defined over $\Q$ only, it was quickly realized how it could be defined over fields of any characteristic. Nevertheless, it took a bit of time before Seed found examples where the $s$-invariant of a knot differs for fields of different characteristic. More such examples have been found, but until recently the only characteristic where there was a difference to the original $s$-invariant was characteristic $2$.

In \cite{LewarkZib} Lewark and Zibrowius found the first example of a knot $K$ such that $s^\Q(K)\not=s^{\F_3}(K)$, where $\F_3$ is the field of three elements. This example is a Whitehead double of the $(3,4)$-torus knot. In fact, a Whitehead double of the $(2,3)$-torus knot has this property over the field of two elements, so one may guess that for $p>3$ the $(p,p+1)$-torus knot gives rise to examples with $s^\Q(K)\not=s^{\F_p}(K)$ via Whitehead doubles. At the moment this seems to be out of reach for computer calculations, and we know of no further examples with $s^\Q(K)\not=s^{\F_p}(K)$ for $p\not=2,3$.

In order to deal with the $s$-invariants over all characteristics it seems natural to try to define an $s$-invariant over $\Z$, and in this paper we propose such.a definition. Unlike over a field, where $s^\F(K)\in 2\Z$, we define $s^\Z(K)\in 2\Z\times \Z^{\gl(K)}$, where $\gl(K)$ is a non-negative integer, and indeed, a knot invariant. In many cases $\gl(K)=0$, and then $s^\Z(K)\in 2\Z$ agrees with $s^\Q(K)$ (in fact, the $2\Z$-factor always coincides with $s^\Q(K)$). If $\gl(K)>0$, the entries in $\Z^{\gl(K)}$ correspond to orders of finite cyclic groups related to a spectral sequence.

Rasmussen's definition is based on a filtration of the Lee complex \cite{MR2173845}, but we will work over the Bar-Natan complex \cite{MR2174270}, see also \cite{MR2320159, MR3189434}. This complex has the advantage that it also works in characteristic $2$, and has a reduced version. The cohomology of the reduced Bar-Natan complex is concentrated in homological degree $0$, where it is of rank $1$. Over a field, the spectral sequence corresponding to the aforementioned filtration ends with exactly one group $E^{0,q}_\infty\not=0$, and this $q$ is the $s$-invariant.

We still have this spectral sequence over $\Z$, but the $E_\infty$-term may not be as well behaved, and there can be several $q$ with $E^{0,q}_\infty\not=0$. Nonetheless, for the largest such $q$ the group is infinite cyclic, while for all the other $q$'s the group is finite cyclic. Our $s^\Z(K)$ essentially just collects this information.

If a prime $p$ does not divide any of these orders of finite groups, we can show that
\begin{equation}\label{eq:ineq}
s^{\F_p}(K) \geq s^\Q(K).
\end{equation}
Using a mirror knot argument we can now easily see that for any knot $K$ we get
\[
s^{\F_p}(K) = s^\Q(K)
\]
for all but finitely many $p$.

One may ask whether (\ref{eq:ineq}) implies that $p$ does not divide any of the orders of the $E^{0,q}_\infty$. It turns out that in general this is not true. In fact, we get the following result.

\begin{theorem}\label{thm:firstmain}
There exists a knot $K$ with $s^\F(K) = 0$ for every field $\F$ and $\gl(K) = 1$. Furthermore, $K$ can be chosen so that its signature is $0$ and its determinant a square.
\end{theorem}

So in general $s^\Z$ is a finer invariant than all of the field coefficient $s$-invariants combined. We also get that $s^\Z$ is a concordance invariant, and the following relation to the slice genus of a knot.

\begin{theorem}\label{thm:secondmain}
Let $K$ be a knot. If $K$ is smoothly slice, then $\gl(K) = 0$. In general, we get a lower bound for the slice genus $g_4(K)$ by
\[
g_4(K)\geq  |s^\Q(K)/2 - \gl(K)|.
\]
\end{theorem}

A slightly different approach to get $s$-invariants over $\Z$ was suggested in \cite[\S 6]{MR3189434}. We believe that our approach essentially captures the information therein, even though \cite{MR3189434} works over the unreduced complex. Furthermore, \cite[Qn 6.2]{MR3189434} asks about the image of the homomorphism from the smooth concordance group $\mathcal{C}$ which sends a knot to all the $s$-invariants over prime fields. As a first approximation to this question we propose a slightly different homomorphism.

For every prime $p$ define a homomorphism $\Sigma_p\colon \mathcal{C}\to \Z$ by 
\[
\Sigma_p(K) =  \frac{s^\Q(K)-s^{\F_p}(K)}{2}.
\]
If $P$ is the set of all prime numbers, we can combine these to  a homomorphism $\Sigma\colon \mathcal{C}\to\bigoplus_{p\in P} \Z$.

\begin{conjecture}
\label{con:bigconj}
The homomorphism $\Sigma$ is surjective.
\end{conjecture}

Evidence for Conjecture \ref{con:bigconj} is scant, but the already mentioned examples of Whitehead doubles of certain torus knots make it seem plausible. At the moment we do not see a viable strategy to prove this conjecture.

The paper is organized as follows. In Section 2 we recall the definition of the Bar-Natan complex and the definition of the $s$-invariants. In Section 3 we consider the reduced version of the Bar-Natan complex, and obtain some splitting results for the unreduced version. In particular, we show that the corresponding unreduced spectral sequence over $\F_2$ splits into twice the reduced version.

In Section 4 we define the integral $s$-invariant and prove its main properties. Section 5 deals with the Lipshitz-Sarkar refinement of $s^{\F_2}$ involving the first Steenrod square, and in the final section we list the results of some extensive calculations, together with a few conjectures.

{\bf Acknowledgements.} the author would like to thank Lukas Lewark for valuable comments and discussions, and information on computations.

\section{The Bar-Natan Complex}

Let $\K$ be a commutative field with $1$, and let
\[
A= \K[x]/\langle x^2-x\rangle.
\]
As in \cite{MR2320159} we have a Frobenius system $(A,\iota, \Delta, \varepsilon)$ given by the co-multiplication $\Delta\colon A \to A\otimes A$
\[
\Delta(1) = 1\otimes x + x \otimes 1 - 1\otimes 1, \hspace{1cm} \Delta(x) = x\otimes x,
\]
the unit $\iota\colon \K \to A$ given by inclusion, and a co-unit $\varepsilon\colon A\to \K$ sending $x$ to $1$ and $1$ to $0$.

Given a link diagram $D$ we can use this Frobenius system to get a cochain complex, compare \cite{MR2232858}, and following \cite{MR3189434} we call it the {\em Bar-Natan complex} $\CBN^\ast(D;\K)$. The underlying cochain groups agree with the usual Khovanov complex $\CKh^\ast(D;\K)$, and are therefore bi-graded. The Bar-Natan differential does not preserve the $q$-grading, but also does not decrease it. We therefore get a filtration 
\[
\cdots \subset \mathcal{F}_{q+2} \subset \mathcal{F}_q\subset \mathcal{F}_{q-2} \subset \cdots \subset \CBN^\ast(D;\K),
\]
where
\[
\mathcal{F}_q = \bigoplus_{j\geq q} \CKh^{\ast, j}(D;\K).
\]
A thorough proof that $\BN^\ast(D;\K)$ is free over $\K$ with a total rank of $2^c$, where $c$ is the number of components of the link can be found in \cite{MR4079621}. Furthermore, for a knot diagram the two generators are in homological degree $0$.

We will need specific generators below, so we are going to describe the generators that go back to \cite{MR2173845}. First note that $A$ is a free $\K$-module of rank $2$, with a basis given by $1, x\in A$. We can get a new basis $x_-, x_\shortmid$ by setting
\[
x_- = x, \hspace{1cm} x_\shortmid = 1 - x.
\]
Multiplication and co-multiplication diagonalize in this basis, see \cite[(2.1)]{MR3189434}, which is the reason the cohomology is so well-behaved, compare \cite{MR4079621}.

 Define $I'\colon A \to A$ by
\[
I'(1) = 1 \hspace{1cm} I'(x) = 1 - x.
\]
In terms of the basis $x_-, x_\shortmid$ this simply permutes these two elements. This is not quite an involution of the Frobenius system $(A, \iota, \Delta, \varepsilon)$, but we get an isomorphism between $(A, \iota, \Delta, \varepsilon)$ and $(A, \iota, \Delta', \varepsilon')$, where
\[
\Delta'(a) = \Delta(-a), \hspace{1cm}\varepsilon'(a) = \epsilon(-a).
\]
This is the result of twisting the original Frobenius system by the unit $-1\in A$. If we denote the corresponding link complex $C^\ast(D;\K)$, we get an isomorphism $\Theta$ between this complex and the Bar-Natan complex, see \cite[Prop.3]{MR2232858}. With this isomorphism we now get an involution $I\colon \CBN^\ast(D;\K) \to \CBN^\ast(D;\K)$. 

In order to say something meaningful about this involution, we need to take a closer look at the isomorphism constructed in \cite{MR2232858}. But restricting this isomorphism to each direct summand $C_S$ corresponding to a smoothing $S$ is just multiplication by a power of $-1$. Furthermore, the isomorphism is obtained by starting with identity on the $0$-smoothing of $D$, then extending over the vertices over the cube. If a smoothing $S'$ is obtained from a smoothing $S$ by a merge, we use the same factor of the identity on $C_{S'}$ that is used on $C_S$, and if it is a split, we multiply the factor by an extra $-1$.


\begin{lemma}\label{lm:pushqup}
Let $D$ be a link diagram. Then there exists a sign $\varepsilon_j\in \{1, -1\}$ which only depends on $j \bmod 4$, such that the chain map
\[
\id+\varepsilon_j I\colon \CBN^\ast(D;\K)\to \CBN^\ast(D;\K)
\]
satisfies
\[
\id+\varepsilon_j I(\mathcal{F}_j)\subset \mathcal{F}_{j+2}.
\]
Furthermore, $\varepsilon_j = -\varepsilon_{j+2}$.
\end{lemma}

\begin{proof}
Clearly $u\pm I(u)\in \mathcal{F}_{j+2}$ for all $u\in \mathcal{F}_{j+2}$, so let us focus on the generators of the $C_S$ with $q$-degree $j$, where $S$ is a smoothing of the diagram. Let $d_S$ be the number of circles in $S$, and $e_S$ be an integer such that $\Theta|C_S = (-1)^{e_S}\id$.

Now let $u=\xi_1\otimes \cdots \otimes \xi_{d_S} \in C_S$ be a generator of $q$-degree $j$, where $\xi_i \in \{1,x\}$. For all of these generators of fixed $q$-degree $j$ the number of $x$ among the $\xi_i$ is the same, so denote this by $f^j_S$. Then
\begin{align*}
I(u) &= (-1)^{e_S} I'(\xi_1)\otimes \cdots\otimes I'(\xi_{d_S}) \\
&= (-1)^{e_S+f^j_S} (\xi_1\otimes \cdots \otimes \xi_{d_S} + v),
\end{align*}
where $v\in \mathcal{F}_{j+2}$. In particular
\[
u + (-1)^{e_S+f^j_S+1} I(u) \in \mathcal{F}_{j+2}.
\]
Every smoothing $S$ can be reached from the all $0$-smoothing $S_0$ by a finite sequence of surgeries. If $S$ is obtained from a smoothing $S'$ by a merge, then $(-1)^{e_S}=(-1)^{e_{S'}}$, and $f^j_S=f^j_{S'}$.

On the other hand, if $S$ is obtained from $S'$ by a split, we get $(-1)^{e_S} = -(-1)^{e_{S'}}$, but $f^j_S = f^j_{S'}+1$. Therefore $(-1)^{e_S+f^j_S+1}$ only depends on $j$, and we can use it for $\varepsilon_j$.

Since $f^j_S = f^{j+2}_S+1$ we get the remaining statements about the dependency of $\varepsilon_j$ on $j$.
\end{proof}

Let $\mathcal{O}$ be an orientation for the link diagram. There is a unique smoothing $S$ of $D$ with every circle in $S$ naturally oriented from $\mathcal{O}$. We can assign a mod $2$ invariant $i(C)$ to each circle $C$ in $S$, which is the number (mod $2$) of other circles in $S$ that separate it from the point at infinity, plus $1$ if the orientation on $C$ is clockwise. 

Now $S$ generates a direct summand $C_S\leq \CBN^\ast(D;\K)$ corresponding to a vertex in the cube. We get $C_S\cong A^{\otimes c}$, where $c$ is the number of circles in $S$, and we can define $s_\mathcal{O}\in C_S$ by
\[
s_\mathcal{O} = x_{C_1}\otimes \cdots \otimes x_{C_c},
\]
where $x_{C_k} = x_-$, if $i(C_k) = 0$, and $x_{C_k} = x_\shortmid$, if $i(C_k) = 1$.

It follows from \cite{MR4079621} that $\BN^\ast(D;\K)$ is freely generated by the cohomology classes of the $s_\mathcal{O}$, where $\mathcal{O}$ runs through the orientations of the link diagram.

We are now going to assume that $D$ is a knot diagram. The filtration on $\CBN^\ast(D;\K)$ induces a filtration on $\BN^0(D;\K)$ given by
\[
\BN^0(D;\K)_q = \im (H^0(\mathcal{F}_q;\K) \to \BN^0(D;\K)).
\]
Clearly, $\BN^0(D;\K)_q = 0$ for large $q$, and $\BN^0(D;\K)_q = \K\oplus \K$, but a crucial property of this filtration for a field $\F$ is that
\[
s^\F_{\max}(D) = \max \{ q \in 2\Z+1 \mid \BN^0(D;\F)_q \not= 0 \}
\]
and
\[
s^\F_{\min}(D) = \max \{ q \in 2\Z+1 \mid  \BN^0(D;\F)_q = \F \oplus \F \}
\]
satisfy
\begin{equation}\label{eq:nicejump}
s^\F_{\max}(D) = s^\F_{\min}(D) + 2,
\end{equation}
and these numbers do not depend on the diagram. A proof for all fields can be found in \cite{MR3189434}, but the result goes back to \cite{MR2729272}. We can therefore define the {\em Rasmussen invariant over $\F$} by
\[
s^{\F}(K) = s^\F_{\max}(D) - 1,
\]
where $K$ is a knot with a diagram $D$. Clearly this invariant only depends on the prime field, and we write $\Q$, and $\F_p$ for a prime number $p$, for these.

\begin{remark}
The proof of (\ref{eq:nicejump}) in \cite{MR3189434} distinguishes the cases of $\chr(\F) = 2$ or not. The case of $\chr(\F) \not= 2$ is reduced to the statement for the Lee-complex using \cite{MR2320159}, while the case $\chr(\F) = 2$ uses the involution $I$. Lemma \ref{lm:pushqup} can be used to extend the characteristic $2$ proof given in \cite{MR3189434} to arbitrary characteristic.
\end{remark}

\section{A reduced Bar-Natan complex}

One advantage of the Bar-Natan complex over the Lee complex is that we can form a reduced version of it. This was already observed in \cite{MR2286127} over $\F_2$, and in \cite[\S 3.3.2]{kotelskiy2019immersed} in general. See also \cite[\S 2]{MR3458146}.

Putting a basepoint on the link diagram $D$ turns $\CBN^\ast(D;\K)$ into a free $A$-cochain complex. The action of $x\in A$ on the direct summand $C_S$ corresponding to a smoothing $S$ is given by multiplying the $A$-factor in the tensor product $A^{\otimes c}=C_S$ belonging to the circle containing the basepoint, by $x$.

We now define
\[
\tCBN^\ast(D;\K) = x\CBN^\ast(D;\K) = \im (x \colon \CBN^\ast(D;\K) \to \CBN^\ast(D;\K)).
\]
\begin{proposition}\label{prp:sesred}
Let $D$ be a link diagram and $\K$ a commutative ring. Then there is a short exact sequence of cochain complexes
\[
0\longrightarrow \tCBN^\ast(D;\K) \stackrel{i}{\longrightarrow} \CBN^\ast(D;\K) \stackrel{T}{\longrightarrow} \tCBN^\ast(D;\K)\longrightarrow 0
\]
where $i$ is inclusion, and $T$ is given by $T(u) = xI(u)$. Furthermore, this sequence splits as cochain complexes over $\K$.
\end{proposition}

Over $\F_2$ this result can also be obtained from \cite{MR3482492}.

\begin{proof}
If $xu\in \tCBN^\ast(D;\K)$, then $I(xu) = (1-x)v$ for some $v\in \CBN^\ast(D;\K)$. Therefore $xI(xu) = x(1-x)v = 0$, and $T\circ i = 0$. To see that $T$ is surjective, define $S\colon \tCBN^\ast(D;\K)\to\CBN^\ast(D;\K)$ by $S(xu) = xu+I(xu)$. Then
\[
T(S(xu)) = xI(xu) +xI(I(xu)) = 0 + x(xu) = xu,
\]
so the sequence splits. It remains to show that $\ker T = \tCBN^\ast(D;\K)$. From the splitting we get that the elements of the kernel are given by $u - xI(u) - I(x I(u))$. We claim that these elements are of the form $xv$ for some $v\in \CBN^\ast(D;\K)$. To see this assume that $u$ is a standard basis element of $\CBN^\ast(D;\K)$. 
We can write $u = 1\otimes s$ or $x\otimes s$, where the first element of the tensor product refers to the component of the smoothing containing the basepoint, and $s\in A^{\otimes k-1}$ where $k$ is the number of components in the smoothing. If $u = 1\otimes s$, then $I(u) = \pm1\otimes I'(s)$, where the sign only depends on the position of the smoothing in the cube. Then
\begin{align*}
u-xI(u) -I(x I(u)) & = u  \pm x \otimes I'(s) - I(\pm x\otimes I'(s)) \\
& = u  \pm x\otimes I'(s) - (1-x)\otimes s  = x\otimes s\pm I'(s).
\end{align*}
If $u = x\otimes s$, then
\[
u-xI(u) - I(x I(u)) = u - 0 - 0 = x\otimes s.
\]
In both cases, $u-xI(u)-I(x I(u))\in \tCBN^\ast(D;\K)$.
\end{proof}

Note that if $\mathcal{O}$ is the choice of an orientation for $D$, then either $s_\mathcal{O}$ is in $\tCBN^\ast(D;\K)$ or $-\mathcal{O}$ is, the orientation obtained by reversing the direction on all components. Since $I(s_\mathcal{O}) = \pm s_{-\mathcal{O}}$, we get for $s_\mathcal{O}\in \tCBN^\ast(D;\K)$ that $T(s_{-\mathcal{O}}) = \pm s_\mathcal{O}$, and the splitting behaves on cohomology the way we expect, namely, we get free $\K$-modules of rank $2^{c-1}$, with $c$ the number of components in $D$.

We get a filtration 
\[
\cdots \subset \tilde{\mathcal{F}}_{q+1} \subset \tilde{\mathcal{F}}_{q-1} \subset \cdots \subset \tCBN^\ast(D;\K),
\]
by setting
\[
\tilde{\mathcal{F}}_{q+1} = x\mathcal{F}_{q+2}.
\]
The short exact sequence of Proposition \ref{prp:sesred} restricts to a short exact sequence
\[
0\longrightarrow \tilde{\mathcal{F}}_{q+1} \longrightarrow \mathcal{F}_q \longrightarrow \tilde{\mathcal{F}}_{q-1} \longrightarrow 0.
\]
If $\K$ has characteristic $2$, the splitting $S\colon \tCBN^\ast(D;\K)\to\CBN^\ast(D;\K)$ given by $S(xu) = xu+I(xu)$ also restricts to $S\colon \tilde{\mathcal{F}}_{q-1} \to \mathcal{F}_q$ by Lemma \ref{lm:pushqup}. We therefore get

\begin{theorem}\label{thm:splitspec}
Let $D$ be a link diagram. Then there is an isomorphism of filtered cochain complexes
\[
(\CBN^\ast(D;\F_2), (\mathcal{F}_q))\cong (\tCBN^\ast(D;\F_2), (\tilde{\mathcal{F}}_{q+1})) \oplus (\tCBN^\ast(D;\F_2), (\tilde{\mathcal{F}}_{q-1})).
\]
In particular, the unreduced spectral sequence over $\F_2$ splits into two copies of the reduced spectral sequence. \hfill\qed
\end{theorem}

To fix our notation, and since we are mainly going to use the reduced version, we write the $n$-th page of the spectral sequence corresponding to $(\tBN^\ast(K;\K), \tilde{\mathcal{F}}_q)$ as $E^{\ast,\ast}_{n, \K}$ (thus omitting the knot from the notation). We use the convention that $E^{p,q}_{1, \K} = \tilde{H}_{\Kh}^{p,q}(K;\K)$.


\begin{definition}
Let $D$ be a knot diagram and $\F$ a field. Define
\[
\tilde{s}^\F(D) = \max\{ q \in 2\Z \mid \tBN^0(D;\F)_q \not = 0\}
\]
\end{definition}

\begin{proposition}\label{prp:reduceds}
Let $K$ be a knot and $\F$ a field. Then
\[
s^\F(K) = \tilde{s}^\F(D)
\]
for all diagrams $D$ of $K$.
\end{proposition}

This is proven in \cite[Prop.3.7]{kotelskiy2019immersed}. We give a proof below as we are going to have use for the techniques later.

\begin{proof}
Let us write $\tilde{s} = \tilde{s}^\F(D)$ and $s = s^\F(K)$. Then $H^0(\tilde{\mathcal{F}}_{\tilde{s}};\F) \to \tBN^0(D;\F)$ is surjective. Since $\tBN^0(D;\F)\cong \F$ is generated by the cohomology class of $s_\mathcal{O}$, there is a cocycle $u\in \tilde{\mathcal{F}}_{\tilde{s}}$ with $u-s_\mathcal{O}$ a coboundary in $\tCBN^\ast(D;\F)$. This means $u\in \mathcal{F}_{\tilde{s}-1}$ is cohomologous to $s_\mathcal{O}$ in $\CBN^\ast(D;\F)$.
As $I$ preserves the filtration, we have $I(u)\in \mathcal{F}_{\tilde{s}-1}$ cohomologous to $s_\mathcal{-O}$ in $\CBN^\ast(D;\F)$.
But this implies that $H^0(\mathcal{F}_{\tilde{s}-1};\F)\to \BN^0(D; \F)$ is surjective, so that $s \geq \tilde{s}$.

We know that $\BN^0(D;\F)$ is freely generated by the cohomology classes of $s_\mathcal{O}$ and $s_{-\mathcal{O}}$. Since $H^0(\mathcal{F}_{s-1};\F)\to \BN^0(D;\F)$ is surjective, there exists a cocycle $u\in \mathcal{F}_{s-1}$ that represents the cohomology class of $s_\mathcal{O}$ in $\BN^0(D;\F)$. By Lemma \ref{lm:pushqup} we have $v = u\pm I(u)\in \mathcal{F}_{s+1}$, and as $I$ is a cochain map, $v$ is a cocycle, which represents the cohomology class of $s_\mathcal{O}\pm s_{-\mathcal{O}}$ in $\BN^0(D;\F)$, a non-zero element. 

Now $xv \in x\mathcal{F}_{s+1} = \tilde{\mathcal{F}}_s$ is cohomologous to $x s_\mathcal{O} \pm x s_{-\mathcal{O}} = s_{\mathcal{O}}$ in $\tCBN^0(D;\F)$. Since the cohomology class of $s_\mathcal{O}$ generates $\tBN^0(D;\F)$, we get $\tilde{s}\geq s$.
\end{proof}

\section{An integral approach to the Rasmussen invariant}

\begin{definition}
Let $C^\ast$ be a cochain complex of abelian groups, that has a filtration
\[
\cdots \subset \mathcal{F}_{q+2} \subset \mathcal{F}_q\subset \mathcal{F}_{q-2} \subset \cdots \subset C^\ast.
\]
The {\em induced filtration}
\[
\cdots H^\ast(C)_{q+2} \subset H^\ast(C)_q \subset H^\ast(C)_{q-2} \subset \cdots \subset H^\ast(C)
\]
on $H^\ast(C)$ is defined by
\[
H^\ast(C)_q = \im (H^\ast(\mathcal{F}_q) \to H^\ast(C))
\]
and the {\em associated graded cohomology groups} $(H^\ast(C)^{(q)})$ are defined by
\[
H^\ast(C)^{(q)} = H^\ast(C)_q / H^\ast(C)_{q+2}.
\]
\end{definition}

Notice that $H^{p}(C)^{(q)} = E^{p,q}_\infty$ for the corresponding spectral sequence.

So in the case of a knot diagram $D$ and a field $\F$ the only non-trivial groups among the associated graded cohomology groups of the Bar-Natan complex $\CBN^\ast(D;\F)$ are given by $\BN^0(D;\F)^{(s^\F\pm 1)}$, and both are $\F$. In the reduced case the only non-zero group is $\tBN^0(D;\F)^{(s^\F)}=\F$.

In the case of the Bar-Natan complex over $\Z$ we can use that $\Q$ is a localization of $\Z$, and therefore $\BN^0(D;\Z)^{(s^\Q\pm 1)}$ are the only ones with non-zero rank. We also get that $\BN^0(D;\Z)^{(q)} = 0$ for $q> s^\Q+1$, but for $q\leq s^\Q-1$ these groups can contain torsion. Because of the spectral sequence this can only occur if $H^{0,q}_{\Kh}(D;\Z)$ is non-zero.

\begin{remark}\label{rem:gradedreduced}
The reduced version is even simpler. We get that $q=s^\Q$ is the largest $q$ for which $H^0(\tilde{\mathcal{F}}_q;\Z) \to \tBN^0(D;\Z)$ is non-zero. Therefore $\tBN^0(D;\Z)^{(q)} = 0$ for $q>s^\Q$,
\[
\tBN^0(D;\Z)^{(s^\Q)} = \Z
\]
with this being the only such group of positive rank. Some of the $\tBN^0(D;\Z)^{(q)}$ for $q < s^\Q$ can be cyclic of finite order, while from some point onward they will be $0$.
\end{remark}

\begin{definition}
Let $C^\ast$ be a cochain complex of abelian groups together with a descending filtration
\[
\cdots \subset \mathcal{F}_{q+2} \subset \mathcal{F}_q\subset \mathcal{F}_{q-2} \subset \cdots \subset C^\ast.
\]
We say that $H^\ast(C^\ast)$ is {\em graded torsion free}, if none of the groups $H^\ast(C)^{(q)}$ contain elements of finite order. If $p$ is a prime number, we say that $H^\ast(C^\ast)$ is {\em graded $p$-torsion free}, if none of the $H^\ast(C)^{(q)}$ contain elements of order $p$.
\end{definition}

In the case of a knot, we would like to get that the graded cohomology groups of the Bar-Natan complex (both reduced and unreduced) do not depend on the diagram. This is indeed the case. To prove this, it appears to be easiest to check the Reidemeister moves via Gauss elimination, as in \cite[\S 9]{MR2320156}. Notice that rather than working with a tangle version, one can simply perform the Gauss eliminations in parallel on the Bar-Natan complex, compare \cite{MR4244204}. Each cancellation is filtration preserving, and therefore we get the desired result.

We can now define our integral $s$-invariant as follows.

\begin{definition}
Let $K$ be a knot. We define the {\em graded length}, $\gl(K)$, to be the largest integer $l$ such that $\tBN^0(K;\Z)^{(s^\Q(K)-2l)}$ is non-zero. The {\em integral $s$-invariant} is the tupel
\[
s^\Z(K) = (s^\Q(K), |\tBN^0(K;\Z)^{(s^\Q(K)-2)}|,\ldots, |\tBN^0(K;\Z)^{(s^\Q(K)-2\gl(K))}|),
\]
where $|G|$ of a finite group $G$ denotes its cardinality.
In case $\gl(K) = 0$ we simply write $s^\Z(K) = s^\Q(K)$.
\end{definition}

\begin{lemma}\label{lm:sinequ}
Let $K$ be a knot such that $\tBN^\ast(K;\Z)$ is graded $p$-torsion free for a prime number $p$. Then
\[
s^{\F_p}(K) \geq s^{\Q}(K).
\]
\end{lemma}

\begin{proof}
Let $\Z_{(p)}$ be the subring of $\Q$ obtained from $\Z$ by inverting all primes different from $p$. As this is a localization of $\Z$, we get that $\tBN^\ast(K;\Z_{(p)})$ is graded torsion free. Also, since $\Q$ is a localization of $\Z_{(p)}$, the only non-zero group among the associated graded cohomology groups is $\tBN^0(K;\Z_{(p)})^{(s^\Q(K))} \cong \Z_{(p)}$, and $H^0(\tilde{\mathcal{F}}_{s^\Q(K)};\Z_{(p)}) \to \tBN^0(K;\Z_{(p)})$ is surjective.

Since there is a surjective ring homomorphism from $\Z_{(p)}$ to $\F_p$ the corresponding map $H^0(\tilde{\mathcal{F}}_{s^\Q(K)};\F_p) \to \tBN^0(D;\F_p)$ is also surjective. This implies
\[
s^{\F_p}(K) \geq s^{\Q}(K),
\]
which is what is claimed.
\end{proof}

\begin{corollary}
Let $K$ be a knot such that both $\tBN^\ast(K;\Z)$ and $\tBN^\ast(\overline{K};\Z)$ are graded $p$-torsion free for a prime number $p$, where $\overline{K}$ is the mirror diagram of $K$. Then
\[
s^{\F_p}(K) = s^{\Q}(K).
\]
\end{corollary}

\begin{proof}
This follows from $s^{\F}(K) = -s^{\F}(\overline{K})$ for any field $\F$.
\end{proof}

\begin{corollary}
Let $K$ be a knot. Then 
\[
s^{\F_p}(K) = s^{\Q}(K).
\]
for all but finally many primes $p$.\hfill\qed
\end{corollary}

\begin{corollary}
Let $K$ be a knot such that $s^\Z(K)=s^\Q(K)$ and $s^\Z(\overline{K})=s^\Q(\overline{K})$. Then
\[
s^{\F_p}(K) = s^{\Q}(K).
\]
for all primes $p$.\hfill\qed
\end{corollary}

We can refine the contrapositive of Lemma \ref{lm:sinequ} by considering the graded length. We omit the proof.

\begin{corollary}
Let $K$ be a knot. Then
\[
\gl(K) \geq \frac{s^\Q(K)-s^{\F_p}(K)}{2}
\]
for every prime $p$.\qed
\end{corollary}

\begin{remark}
One can show that $\tCBN^\ast(\overline{D};\Z)$ is isomorphic to the dual complex of $\tCBN^\ast(D;\Z)$ in a filtration preserving way, but this does not imply that one is graded torsion free if and only if the other is. Indeed, consider the knot $K = 14^n_{19265}$ in knotscape notation. Figure \ref{fig:14n19265} contains the reduced integral Khovanov homology in quantum degrees $-8$ to $2$.
\begin{figure}[ht]
\begin{tabular}{|c||c|c|c|c|c|c|c|c|}
\hline
\backslashbox{\!$q$\!}{\!$h$\!} & $-4$ & $-3$ & $-2$ & $-1$ & $0$ & $1$ & $2$ & $3$ \\
\hline
\hline
$2$  &   &   &   &   &   &   & $ \Rone $ & $ \Rmor{2} $ \\
\hline
$0$  &   &   &   &   & $ \Rone $ & $ \Rone $ & $ \Rone $ &   \\
\hline
$-2$  &   &   &   &   & $ \Tmor{2}{2} $ & $ \Rone $ &   &   \\
\hline
$-4$  &   &   & $ \Rmor{2} $ & $ \Rone \oplus \Tone{2} $ & $ \Rone $ &   &   &   \\
\hline
$-6$  &   & $ \Rmor{3} $ & $ \Rone \oplus \Tone{2} $ &   &   &   &   &   \\
\hline
$-8$  & $ \Rmor{2} $ & $ \Tmor{2}{2} $ &   &   &   &   &   &   \\
\hline
\end{tabular}
\caption{\label{fig:14n19265}The reduced integral Khovanov homology of $14^n_{19265}$ in quantum degrees $-8$ to $2$.}
\end{figure}
No other quantum degree has non-zero groups in homological degrees between $-3$ and $3$. The $\Z$-summand in bi-degree $(-2,-6)$ has to cancel the $\Z$-summand in bi-degree $(-1,-4)$. Therefore one of the $\Z/2\Z$ summands in bi-degree $(0,-2)$ survives to the $E_\infty$ stage of the spectral sequence, and $\tBN^0(K;\Z)^{(-2)} = \Z/2\Z$.

On the other hand, consider the reduced Khovanov homology of $\overline{K}$ in Figure \ref{fig:themirror}.
\begin{figure}[ht]
\begin{tabular}{|c||c|c|c|c|c|c|c|c|}
\hline
\backslashbox{\!$q$\!}{\!$h$\!} & $-3$ & $-2$ & $-1$ & $0$ & $1$ & $2$ & $3$ & $4$ \\
\hline
\hline
$8$  &   &   &   &   &   &   &   & $ \Rmor{2} \oplus \Tmor{2}{2} $ \\
\hline
$6$  &   &   &   &   &   & $ \Rone $ & $ \Rmor{3} \oplus \Tone{2} $ &   \\
\hline
$4$  &   &   &   & $ \Rone $ & $ \Rone $ & $ \Rmor{2} \oplus \Tone{2} $ &   &   \\
\hline
$2$  &   &   & $ \Rone $ &   & $ \Tmor{2}{2} $ &   &   &   \\
\hline
$0$  &   & $ \Rone $ & $ \Rone $ & $ \Rone $ &   &   &   &   \\
\hline
$-2$  & $ \Rmor{2} $ & $ \Rone $ &   &   &   &   &   &   \\
\hline
\end{tabular}
\caption{\label{fig:themirror}The reduced integral Khovanov homology of the mirror of $14^n_{19265}$ in quantum degrees $-2$ to $8$.}
\end{figure}

The $\Z$ in bi-degree $(1,4)$ has to cancel the $\Z$ in bi-degree $(2,6)$, and so the $\Z$ in bi-degree $(0,0)$ survives to the $E_\infty$ term. By Remark \ref{rem:gradedreduced} the term in bi-degree $(0,4)$ will not survive to the $E_\infty$-stage.

The knot $K$ is of course the first knot (in knotscape ordering) for which $s^{\F_2}(K) \not= s^\Q(K)$, see \cite{MR3189434}.
\end{remark}

The simplest criterion to get $\BN^\ast(K;\Z)$ graded torsion free for a knot is to have only one $q$-degree with non-trivial $\tilde{H}_{\Kh}^{0,j}(K;\Z)$. 

\begin{lemma}\label{lm:easytorfree}
Let $K$ be a knot such that there is exactly one $q$ with $\tilde{H}_{\Kh}^{0,q}(K;\Z)$. Then $\tBN^\ast(K;\Z)$ is graded torsion free, and $q=s^\Q(K)$.\hfill\qed
\end{lemma}

For knots with small crossing numbers the condition of Lemma \ref{lm:easytorfree} is often satisfied, but maybe not as often as one might hope. Still, one can often analyze the spectral sequence to conclude that the Bar-Natan complex is graded torsion free.

Let us introduce the following notation to study this further.

\begin{definition}
Let $K$ be a knot, $n\geq 1$ and $k\leq l$ be integers. We say that $K$ is {\em $n$-thin in degrees $k$ to $l$}, if all non-trivial reduced Khovanov homology groups $\tilde{H}_{\Kh}^{i,q}(K;\Z)$ with $k\leq i \leq l$ lie on $n$ adjacent diagonals, and are torsion free.  We say that $K$ is {\em $n$-thin}, if $K$ is $n$-thin in degrees $-c$ to $c$, where $c$ is the crossing number of $K$.
\end{definition}

\begin{lemma}
Let $K$ be a knot which is $2$-thin in degrees $-1$ to $2$ and whose reduced Khovanov homology is also torsion free in homological degree $3$, and $p$ a prime number. If $K$ is not graded $p$-torsion free, then $s^\Q(K) =s^{\F_p}(K)+2$.
\end{lemma}

\begin{proof}
Let $q$ and $q+2$ be the only $q$-degrees which have non-trivial reduced Khovanov homology in homological degree $0$. For the spectral sequence coming from the filtration we write $E^{i,j}_1 = \tilde{H}_{\Kh}^{i,j}(K;\Z)$. By assumption, $E_1^{0,q}$ is torsion free, and $E_3^{0,q}=E_\infty^{0,q}$. The $p$-torsion from the assumption has to sit in $E_\infty^{0,q}$. If $E_2^{0,q}$ does not contain $p$-torsion, neither will $E^{0,q}_\infty$, as $E_2^{-1,q-2} = 0$ by assumption. So $E_2^{0,q}$ contains $p$-torsion that survives to $E_\infty^{0,q}$.

Let us write $E_{n,\F_p}^{i,j}$ for the corresponding spectral sequence with coefficients in $\F_p$. Since the reduced Khovanov homology is torsion free in homological degrees $-1$ to $3$, we get
\begin{align*}
E_{2,\F_p}^{0,q} &= E_2^{0,q}\otimes \F_p \oplus \Tor(E_2^{1,q+2}, \F_p),\\
E_{2,\F_p}^{1,q+2} &= E_2^{1,q+2}\otimes \F_p \oplus \Tor(E_2^{2,q+4}, \F_p),
\end{align*}
by the Universal Coefficient Theorem. Furthermore, the next differential $d_2^{\F_p}$ is just $d_2\otimes \id$ on the tensor product summands. In particular, the $p$-torsion of $E_2^{0,q}$ leads to a surviving $\F_p$-summand in $E_{\infty, \F_p}^{0,q}$. Since we have $E_\infty^{0,q+2}=\Z$ by Remark \ref{rem:gradedreduced}, and $\F_p$ is a field, we get $s^\Q(K) = q+2$, while $s^{\F_p} = q$.
\end{proof}

The only knot with up to $13$ crossings which is not $2$-thin in degrees $-1$ to $3$ is $13^n_{3663}$. This knot satisfies the conditions of Lemma \ref{lm:easytorfree}, although its mirror does not.

We do not have an example of a knot for which $\tilde{H}_{\Kh}^{\ast}(K)$ is torsion-free, but $\tBN^0(K)$ not graded torsion-free. Indeed, with knots up to $16$ crossings torsion in reduced Khovanov homology is rare, and for these knots we have both $s^\Z(K)=s^\Q(K)$ and  $s^\Z(\overline{K})=s^\Q(\overline{K})$, unless $s^\Q(K)\not=s^{\F_2}(K)$.

Torsion can occur in later stages of the spectral sequence. For example, for the torus knot $T(8,9)$ we get that the spectral sequence over $\Q$ collapses at a different page than the spectral sequence over $\F_7$ (both reduced and unreduced), despite the Khovanov homology not containing any $7$-torsion. Of course, this stays away from homological degree $0$ in this case.

The next example is the easiest way to produce algebraically a filtered cochain complex $C^\ast$ whose cohomology is supported in homological degree $0$ only, with this group being $\Z$, and such that it is not graded torsion-free. We suspect that most of our examples contain this complex (up to filtered chain homotopy equivalence) in their reduced Bar-Natan complex as a direct summand. For knots $K$ with up to 15 crossings and $s^{\F_2}(K)\not=s^\Q(K)$ this has been checked computationally, see \cite[Rm.4.5]{iltgen2021khovanov}.

\begin{example}\label{ex:staircase}
Let $p$ be a prime number and $C^\ast$ be the cochain complex with generators and coboundary as follows:
\[
\begin{tikzpicture}
\node at (0,0) {$\Z$};
\node at (2,0) {$\Z$};
\node at (2,1) {$\Z$};
\draw[->] (0.2,0) -- node [above, near end] {$p$} (1.8,0);
\draw[->] (0.2, 0.1) -- node [above, near end] {$1$} (1.8, 1);
\draw[-, dashed] (-0.4, 0.5) -- (2.4, 0.5);
\end{tikzpicture}
\]
The left-most copy of $\Z$ is in homological degree $-1$, and the two copies of $\Z$ in the middle form $C^0 = \Z^2$. The filtration is indicated by the dashed line, with $\mathcal{F}_0$ only containing one copy of $\Z$, with $\mathcal{F}_{-2} = C^\ast$.

The $E_1$-term of the spectral sequence has only two non-zero groups, namely $E_1^{0,0} = \Z$ and  $E_1^{0,-2} = \Z/p\Z$. No further cancellations are possible, so we have graded $p$-torsion here.
\end{example}

One might think that it is easier for the reduced Bar-Natan complex to be graded torsion free than for the unreduced version, but this is not the case.

\begin{lemma}\label{lm:redunred}
Let $K$ be a knot and $p$ a prime number. Then $\tBN^\ast(K;\Z)$ is graded $p$-torsion free if and only if $\BN^\ast(K;\Z)$ is graded $p$-torsion free.
\end{lemma}

\begin{proof}
Note that $\BN^\ast(K;\Z)$ is graded $p$-torsion free if and only if $\BN^\ast(K;\Z_{(p)})$ is graded torsion free, and similarly in the reduced case. We therefore work over $\Z_{(p)}$ and omit coefficients.

Assume that $\tBN^\ast(K)$ is graded torsion free. Then $H^0(\tilde{\mathcal{F}}_{s^\Q})\to \tBN^0(K)$ is surjective, and there is a cocycle $u\in \tilde{\mathcal{F}}_{s^\Q}$ which is cohomologous to $s_\mathcal{O}$ in $\tCBN^0(K)$. We now view $u, I(u) \in \mathcal{F}_{s^\Q-1}$ and get that additionally $I(u)$ is cohomologous to $s_{-\mathcal{O}}$ in $\CBN^0(K)$. In particular, $H^0(\mathcal{F}_{s^\Q-1})\to \BN^0(K)$ is surjective.

By Lemma \ref{lm:pushqup} $u\pm I(u)\in \mathcal{F}_{s^\Q+1}$, and therefore
\[
\BN^0(K)^{(s^\Q-1)} \cong \Z_{(p)} \cong \BN^0(K)^{(s^\Q+1)}.
\]
That $\BN^\ast(K)$ graded torsion free implies $\tBN^\ast(K)$ is graded torsion free is similar, but easier, and left to the reader.
\end{proof}

\begin{theorem}\label{thm:sliceinv}
$s^\Z$ is a concordance invariant.
\end{theorem}

\begin{proof}
Let $C$ be a cylinder between the knots $K$ and $K'$ which makes them concordant. By \cite[Thm.2.3, Prop.2.4]{MR3189434} there exist filtration preserving chain maps
\[
\varphi\colon \CBN^\ast(K;\Z)\to \CBN^\ast(K';\Z) \mbox{ and } \psi\colon \CBN^\ast(K';\Z)\to \CBN^\ast(K;\Z)
\]
which induce isomorphisms on cohomology. By picking a path of basepoints on $C$ that stays away from births and deaths, we also get this statement on the reduced Bar-Natan complexes. 

Consider the diagram
\[
\begin{tikzpicture}
\node at (0,0) {$\tBN^\ast(K;\Z)$};
\node at (4,0) {$\tBN^\ast(K';\Z)$};
\node at (0,1.5) {$H^0(\tilde{\mathcal{F}}_q(K);\Z)$};
\node at (4,1.5) {$H^0(\tilde{\mathcal{F}}_q(K');\Z)$};
\draw[->] (1.1,1.6) -- node [above] {$\varphi$} (2.8,1.6);
\draw[<-] (1.1, 1.4) -- node [below] {$\psi$} (2.8, 1.4);
\draw[->] (0,1.2) -- (0,0.3);
\draw[->] (4,1.2) -- node [right] {$i_K$} (4, 0.3);
\draw[->] (0.9,0) -- node [above] {$\cong$} (3.05,0);
\end{tikzpicture}
\]
After identifying $\tBN^\ast(K;\Z)\cong \Z \cong \tBN^\ast(K';\Z)$ we get from $\phi$ that $\tBN^0(K;\Z)_q\subset \tBN^0(K';\Z)_q$, and from $\psi$ that $\tBN^0(K';\Z)_q\subset \tBN^0(K;\Z)_q$. Hence the associated graded cohomology groups $\tBN^0(K;\Z)^{(q)}\cong \tBN^0(K';\Z)^{(q)}$ for all $q$. The result follows.
\end{proof}

%
%

A slight variation of this argument gives a lower bound on the smooth slice genus $g_4(K)$, taking the graded length into account. We omit the details.

\begin{theorem}\label{thm:genusbound}
Let $K$ be a knot. Then
\[
\pushQED{\qed} 
g_4(K) \geq |s^\Q(K)/2 - \gl(K)|. \qedhere
\popQED
\]
\end{theorem}

Theorems \ref{thm:sliceinv} and \ref{thm:genusbound} combine to a proof of Theorem \ref{thm:secondmain}.

\section{The Lipshitz-Sarkar refinements of Rasmussen's invariant}

In \cite{MR3189434} Lipshitz and Sarkar showed how stable cohomology operations on singular cohomology give rise to refinements of the various $s$-invariants. While they were able to make interesting calculations for the second Steenrod square, more effective calculations can so far only be done on Bockstein homomorphisms, with only the first Steenrod square being implemented thus far.

To describe these refinements, assume that $\alpha\colon \tilde{H}^\ast(\cdot;\F) \to \tilde{H}^{\ast+n_\alpha}(\cdot; \F)$ is a stable cohomology operation on singular cohomology with coefficients in a field $\F$. The suspension spectrum defined in \cite{MR3230817} turns this into a cohomology operation $\alpha\colon H_{\Kh}^{i,q}(K;\F)\to H_{\Kh}^{i+n_\alpha,q}(K,\F)$ for every $q$. Therefore we have a zig-zag of maps
\[
 H_{\Kh}^{-n_\alpha,q}(K;\F) \stackrel{\alpha}{\longrightarrow} H_{\Kh}^{0,q}(K;\F) \longleftarrow H^0(\mathcal{F}_q;\F) \longrightarrow \BN^0(K;\F)\cong \F^2.
\]
Consider the following configurations related to this zig-zag:
\begin{equation}\label{eq:fullconfigs}
\begin{tikzpicture}[baseline={([yshift=-.5ex]current bounding box.center)}]
\node at (0,0) {$\langle \tilde{a},\tilde{b}\rangle$};
\node at (3,0) {$\langle \hat{a}, \hat{b} \rangle$};
\node at (6,0) {$\langle a,b\rangle$};
\node at (9,0) {$\langle \bar{a},\bar{b}\rangle$};
\node at (0,1) {$H_{\Kh}^{-n_\alpha,q}(K;\F)$};
\node at (3,1) {$H_{\Kh}^{0,q}(K;\F)$};
\node at (6,1) {$H^0(\mathcal{F}_q;\F)$};
\node at (9,1) {$\BN^0(K;\F)$};
\node at (0,2) {$\langle \tilde{a}\rangle$};
\node at (3,2) {$\langle \hat{a}\rangle$};
\node at (6,2) {$\langle a\rangle$};
\node at (9,2) {$\langle \bar{a}\rangle\not=0$};
\draw[right hook->] (0,0.25) -- (0,0.75);
\draw[right hook->] (3,0.25) -- (3,0.75);
\draw[right hook->] (6,0.25) -- (6,0.75);
\draw[-] (9.05,0.25) -- (9.05,0.75);
\draw[-] (8.95,0.25) -- (8.95,0.75);
\draw[right hook->] (0,1.75) -- (0,1.25);
\draw[right hook->] (3,1.75) -- (3,1.25);
\draw[right hook->] (6,1.75) -- (6,1.25);
\draw[right hook->] (9,1.75) -- (9,1.25);
\draw[->] (0.5,0) -- (2.5,0);
\draw[<-] (3.5,0) -- (5.5,0);
\draw[->] (6.5,0) -- (8.5,0);
\draw[->] (0.4,2) -- (2.6,2);
\draw[<-] (3.4,2) -- (5.6,2);
\draw[->] (6.4,2) -- (8.3,2);
\draw[->] (1.1,1) -- node[above] {$\alpha$} (2.1,1);
\draw[<-] (3.95,1) -- node[above] {$p$} (5.15,1);
\draw[->] (6.85,1) -- node[above] {$i$} (8.1,1);
\end{tikzpicture}
\end{equation}

\begin{definition}
Call an odd integer $q$ \em $\alpha$-half-full\em, if there exist $a\in H^0(\mathcal{F}_q;\F)$ and $\tilde{a}\in H_{\Kh}^{-n_\alpha,q}(K;\F)$ such that $p(a)=\alpha(\tilde{a})$, and such that $i(a)=\bar{a}\not=0$. That is, there exists a configuration as in the upper two rows of (\ref{eq:fullconfigs}).

Call an odd integer $q$ \em $\alpha$-full\em, if there exist $a,b\in H^0(\mathcal{F}_q;\F)$ and $\tilde{a},\tilde{b}\in H_{\Kh}^{-n_\alpha,q}(K;\F)$ such that $p(a)=\alpha(\tilde{a})$, $p(b)=\alpha(\tilde{b})$, and $i(a),i(b)$ generate $\BN^0(K;\F)$. That is, there exists a configuration as in the lower two rows of (\ref{eq:fullconfigs}).
\end{definition}

Lipshitz and Sarkar \cite{MR3189434} define their refinements of the $s$-invariant as
\begin{definition}
Let $K$ be a knot and $\alpha$ a stable cohomology operation on singular cohomology with coefficients in $\F$, then $r^\alpha_+,r^\alpha_-, s^\alpha_+, s^\alpha_-\in \Z$ are defined as follows.
\begin{align*}
r^\alpha_+(K) &= \max\{q\in 2\Z+1\,|\, q \mbox{ is $\alpha$-half-full}\}+1\\
s^\alpha_+(K) &= \max\{q\in 2\Z+1\,|\, q \mbox{ is $\alpha$-full}\}+3.
\end{align*}
If $\overline{K}$ denotes the mirror of $K$, we also set
\begin{align*}
r^\alpha_-(K) &= -r^\alpha_+(\overline{K})\\
s^\alpha_-(K) &= -s^\alpha_+(\overline{K}).
\end{align*}
\end{definition}

Since \cite{MR3230817} contains a suspension spectrum for reduced Khovanov homology, these definitions carry over to the reduced $s$-invariant. 

\begin{definition}
Let $K$ be a knot. Call an odd integer $q$ {\em reduced $\alpha$-full}, if here exist $a\in H^0(\tilde{\mathcal{F}}_q;\F)$ and $\tilde{a}\in \tilde{H}_{\Kh}^{-n_\alpha,q}(K;\F)$ such that $\tilde{p}(a)=\alpha(\tilde{a})$, and such that $\tilde{\imath}(a)\not=0$, where $\tilde{p}\colon H^0(\tilde{\mathcal{F}}_q;\F)\to \tilde{H}_{\Kh}^{0,q}(K;\F)$ and $\tilde{\imath}\colon H^0(\tilde{\mathcal{F}}_q;\F)\to \tBN^0(K;\F)$.
 
We set
\begin{align*}
\tilde{s}^\alpha_+(K) &= \max\{q\in 2\Z \mid q \mbox{ is reduced $\alpha$-full}\}+2 \\
\tilde{s}^\alpha_-(K) &= -\tilde{s}^\alpha_+(\overline{K}).
\end{align*}
\end{definition}

It is shown in \cite{MR3189434} that $s^\alpha_+(K), r^\alpha_+(K)\in \{s^\F(K), s^\F(K)+2\}$, and the proof carries over to show that $\tilde{s}^\alpha_+(K)\in \{s^\F(K), s^\F(K)+2\}$. The interest in these refinements comes from the fact that they give lower bounds on the slice genus of a knot, so we introduce the following notation.

\begin{definition}
Let $K$ be a knot. We say that $s^\alpha_+$, resp.\ $r^\alpha_+$, resp.\ $\tilde{s}^\alpha_+$, is {\em non-trivial for $K$}, if $s^\alpha_+(K)\not= s^\F(K)$, resp.\ $r^\alpha_+(K)\not= s^\F(K)$, resp.\ $\tilde{s}^\alpha_+(K)\not= s^\F(K)$.
\end{definition}

We note that non-triviality of any of these invariants need not improve the lower bound on the slice genus, but triviality will definitely not improve this lower bound from $s^\F$.

\begin{proposition}\label{prp:nontrvop}
Let $K$ be a knot, and $\alpha$ a stable cohomology operation.
\begin{enumerate}
\item If $r^\alpha_+$ is non-trivial for $K$, then $\tilde{s}^\alpha_+$ is non-trivial for $K$.
\item If $\tilde{s}^\alpha_+$ is non-trivial for $K$, then $s^\alpha_+$ is non-trivial for $K$.
\end{enumerate}
\end{proposition}

\begin{proof}
We have the following diagram:
\begin{equation}\label{eq:bigdiag}
\begin{tikzpicture}[baseline={([yshift=-.5ex]current bounding box.center)}]
\node at (0,0) {$\tilde{H}_{\Kh}^{-n_\alpha,q+1}(K;\F)$};
\node at (3,0) {$\tilde{H}_{\Kh}^{0,q+1}(K;\F)$};
\node at (6,0) {$H^0(\tilde{\mathcal{F}}_{q+1};\F)$};
\node at (9,0) {$\tBN^0(K;\F)$};
\node at (0,1.5) {$H_{\Kh}^{-n_\alpha,q}(K;\F)$};
\node at (3,1.5) {$H_{\Kh}^{0,q}(K;\F)$};
\node at (6,1.5) {$H^0(\mathcal{F}_q;\F)$};
\node at (9,1.5) {$\BN^0(K;\F)$};
\node at (0,3) {$\tilde{H}_{\Kh}^{-n_\alpha,q-1}(K;\F)$};
\node at (3,3) {$\tilde{H}_{\Kh}^{0,q-1}(K;\F)$};
\node at (6,3) {$H^0(\tilde{\mathcal{F}}_{q-1};\F)$};
\node at (9,3) {$\tBN^0(K;\F)$};
\draw[->] (1.3,0) -- node[above] {$\alpha$} (1.95,0);
\draw[<-] (4.05,0) -- node[above] {$\tilde{p}$} (5,0);
\draw[->] (7,0) -- node[above] {$\tilde{\imath}$} (8.1,0);
\draw[->] (1.1,1.5) -- node[above] {$\alpha$} (2.1,1.5);
\draw[<-] (3.95,1.5) -- node[above] {$p$} (5.15,1.5);
\draw[->] (6.85,1.5) -- node[above] {$i$} (8.1,1.5);
\draw[->] (1.3,3) -- node[above] {$\alpha$} (1.95,3);
\draw[<-] (4.05,3) -- node[above] {$\tilde{p}$} (5,3);
\draw[->] (7,3) -- node[above] {$\tilde{\imath}$} (8.1,3);
\draw[->] (0,0.3) -- (0, 1.2);
\draw[->] (3,0.3) -- (3, 1.2);
\draw[->] (6,0.3) -- (6, 1.2);
\draw[->] (9,0.3) -- (9, 1.2);
\draw[->] (0,1.8) -- (0, 2.7);
\draw[->] (3,1.8) -- (3, 2.7);
\draw[->] (6,1.8) -- (6, 2.7);
\draw[->] (9,1.8) -- (9, 2.7);
\end{tikzpicture}
\end{equation}
The third column is from the long exact sequence arising from Proposition \ref{prp:sesred}. In particular, the map $H^0(\mathcal{F}_q;\F)\to H^0(\tilde{\mathcal{F}}_{q-1};\F)$ is induced by $u\mapsto xI(u)$. But in $H^0(\tilde{\mathcal{F}}_{q-1}/\tilde{\mathcal{F}}_{q+1};\F)$ we have $[xI(u)] = \pm [xu]$ by Lemma \ref{lm:pushqup}. So up to sign, this agrees with the usual long exact sequence of reduced and unreduced Khovanov homology, and we can treat the whole diagram as commutative by \cite[\S 9]{MR3230817} (which shows that the stable cohomology operation part commutes).

Now let $s = s^\F(K)$, and assume that $r^\alpha_+$ is non-trivial for $K$. This means that $s+1$ is $\alpha$-half-full, so there exists $a\in H^0(\mathcal{F}_{s+1};\F)$ with $p(a)$ in the image of $\alpha$, and $i(a)$ non-zero in $\BN^0(K;\F)$. In particular, $i(a)$ is a non-zero multiple of $[s_\mathcal{O}\pm s_{-\mathcal{O}}]$, which remains non-zero in $\tBN^0(K;\F)$.

Using (\ref{eq:bigdiag}) with $q = s+1$, we see that $s$ is reduced $\alpha$-full, and therefore $\tilde{s}^\alpha_+$ is non-trivial for $K$.

If we assume that $\tilde{s}^\alpha_+$ is non-trivial for $K$, consider (\ref{eq:bigdiag}) with $q = s-1$. Since $s$ is reduced $\alpha$-full, we can find $a\in H^0(\mathcal{F}_{s-1};\F)$ such that $i(a)\in \BN^0(K;\F)$ is represented by $s_\mathcal{O}$, and $p(a)$ is in the image of $\alpha$. Now choose $b\in H^0(\mathcal{F}_{s-1};\F)$ to be in the image of $H^0(\mathcal{F}_{s+1};\F)$, and being represented by $s_\mathcal{O}\pm s_{-\mathcal{O}}$ in $\BN^0(K;\F)$. Then $i(a),i(b)$ generate $\BN^0(K;\F)$, and $p(b)=0\in H_{\Kh}^{0,s-1}(K;\F)$, which means that $s-1$ is $\alpha$-full. Hence $s^\alpha_+$ is non-trivial for $K$. 
\end{proof}

Because of Proposition \ref{prp:nontrvop} we can think of $s^\alpha_+$ as the strongest of these invariants. Calculations in \cite{MR3189434} show that $s^{\Sq^2}_+$ can be non-trivial while $\tilde{s}^{\Sq^2}_+$ is trivial for some knots. The reason is that the reduced Khovanov homology is not big enough to support a non-trivial second Steenrod square for these knots.

For $\Sq^1$ we do not know any example where $s^{\Sq^1}_+$ is non-trivial, while $\tilde{s}^{\Sq^1}_+$ is trivial. This is a bit surprising, since unreduced Khovanov homology contains much more $2$-torsion than reduced Khovanov homology. But as we will see, this torsion tends to not survive during the stages of the spectral sequence.

From the short exact sequence of reduced and unreduced Khovanov complexes
\[
0\longrightarrow \tilde{\mathcal{F}}_{q+1}/\tilde{\mathcal{F}}_{q+3} \longrightarrow \mathcal{F}_q/\mathcal{F}_{q+2}\longrightarrow \tilde{\mathcal{F}}_{q-1}/\tilde{\mathcal{F}}_{q+1} \longrightarrow 0
\]
we get a connecting homomorphism $\beta\colon \tilde{H}_{\Kh}^{n,q-1}(K;\Z)\to \tilde{H}_{\Kh}^{n+1, q+1}(K;\Z)$. We have another short exact sequence given by
\[
0\longrightarrow \tilde{\mathcal{F}}_{q+1}/\tilde{\mathcal{F}}_{q+3} \longrightarrow \tilde{\mathcal{F}}_{q-1}/\tilde{\mathcal{F}}_{q+3} \longrightarrow \tilde{\mathcal{F}}_{q-1}/\tilde{\mathcal{F}}_{q+1} \longrightarrow 0
\]
whose connecting homomorphism is $d_1\colon E_{1,\Z}^{n,q-1}\to E_{1,\Z}^{n+1,q+1}$, the first differential in the reduced spectral sequence with $\Z$-coefficients.

\begin{lemma} \label{lm:firstdiff}
We have
\[
\beta = \pm 2\,d_1\colon \tilde{H}_{\Kh}^{n,q-1}(K;\Z)\to \tilde{H}_{\Kh}^{n+1, q+1}(K;\Z).
\]
\end{lemma}

\begin{proof}
Consider the following ladder.
\[
\begin{tikzpicture}
\node at (0,0) {$0$};
\draw[->] (0.3, 0) -- (1, 0);
\node at (2,0) {$\tilde{\mathcal{F}}_{q+1}/\tilde{\mathcal{F}}_{q+3}$};
\draw[->] (3,0) -- (4,0);
\node at (5,0) {$\tilde{\mathcal{F}}_{q-1}/\tilde{\mathcal{F}}_{q+3} $};
\draw[->] (6,0) -- (7,0);
\node at (8,0) {$\tilde{\mathcal{F}}_{q-1}/\tilde{\mathcal{F}}_{q+1} $};
\draw[->] (9,0) -- (9.7,0);
\node at (10,0) {$0$};
\node at (0,2) {$0$};
\draw[->] (0.3, 2) -- (1, 2);
\node at (2,2) {$\tilde{\mathcal{F}}_{q+1}/\tilde{\mathcal{F}}_{q+3}$};
\draw[->] (3,2) -- (4.2,2);
\node at (5,2) {$\mathcal{F}_q/\mathcal{F}_{q+2}$};
\draw[->] (5.8,2) -- (7,2);
\node at (8,2) {$\tilde{\mathcal{F}}_{q-1}/\tilde{\mathcal{F}}_{q+1} $};
\draw[->] (9,2) -- (9.7,2);
\node at (10,2) {$0$};
\draw[->] (2,1.7) -- node [right] {$=$} (2, 0.3);
\draw[->] (5, 1.7) -- node [right] {$J$} (5, 0.3);
\draw[->] (8, 1.7) -- node [right] {$\tilde{J}$} (8, 0.3);
\end{tikzpicture}
\]
where $J$ is given by $J(u) = x(u+\varepsilon_{q+2}I(u))$, with $\varepsilon_{q+2}\in \{\pm 1\}$ as in Lemma \ref{lm:pushqup}. In particular, for $u\in \mathcal{F}_{q+2}$ we have $J(u) \in \tilde{\mathcal{F}}_{q+3}$, so that $J$ is well defined.

If $xu\in \tilde{\mathcal{F}}_{q+1}$, then $J(xu) = xu+\varepsilon_{q+2}xI(xu) = xu$, and the left square commutes. The map $\mathcal{F}_q/\mathcal{F}_{q+2}\to \tilde{\mathcal{F}}_{q-1}/\tilde{\mathcal{F}}_{q+1}$ is induced by $T\colon \mathcal{F}_q\to\tilde{\mathcal{F}}_{q-1}$ so sends $u$ to $xI(u)$. As $u-\varepsilon_{q+2} I(u)\in \mathcal{F}_{q+2}$ for $u\in \mathcal{F}_q$, we get $J(u)$ represents $2xu\in \tilde{\mathcal{F}}_{q-1}/\tilde{\mathcal{F}}_{q+1}$, and $T(u)$ represents $\varepsilon_{q+2} xu\in \tilde{\mathcal{F}}_{q-1}/\tilde{\mathcal{F}}_{q+1}$.
So we make the right square commute by setting $\tilde{J}(xu) = 2\,\varepsilon_{q+2} xu$. The result now follows from the naturality of the connecting homomorphisms.
\end{proof}

\begin{remark}
Lemma \ref{lm:firstdiff} gives a reason why 2-torsion is so prevalent in unreduced Khovanov homology. It can also be used to show that $1$-thin links do not contain torsion of order $2^k$ with $k>1$, recovering a theorem of Shumakovitch \cite{shumakovitch2018torsion}.
\end{remark}

\begin{theorem}
Let $K$ be a knot such that $\tilde{H}_{\Kh}^{0,q}(K;\Z)$ does not contain $2$-torsion for $q=s, s-2$, where $s=s^{\F_2}(K)$. Then $s^{\Sq^1}_+(K) = s$.
\end{theorem}

\begin{proof}
Let $B = \im(\Sq^1\colon H_{\Kh}^{-1,s-1}(K;\F_2) \to H_{\Kh}^{0,s-1}(K;\F_2))$. Since $\Sq^1$ is simply a Bockstein homomorphism, naturality of Bockstein homomorphisms implies that $B=p(\overline{B})$, where $p\colon H_{\Kh}^{0,s-1}(K;\Z)\to H_{\Kh}^{0,s-1}(K;\F_2)$ is change of coefficients, and $\overline{B}$ is the subgroup of $H_{\Kh}^{0,s-1}(K;\Z)$ generated by all $\Z/2\Z$ direct summands in $H^{0,s-1}_{\Kh}(K;\Z)$.

Consider the long exact sequence
\[
\tilde{H}_{\Kh}^{-1,s-2}(K;\Z) \stackrel{\beta}{\longrightarrow} \tilde{H}_{\Kh}^{0,s}(K;\Z) \stackrel{i}{\longrightarrow}H_{\Kh}^{0,s-1}(K;\Z)\longrightarrow \tilde{H}_{\Kh}^{0,s-2}(K;\Z)
\]
and some $b\in \overline{B}$. Since $\tilde{H}_{\Kh}^{0,s-2}(K;\Z)$ does not contain elements of order $2$, we get $b = i(\tilde{b})$ for some $\tilde{b}\in \tilde{H}_{\Kh}^{0,s}(K;\Z)$. Then $2\tilde{b}\in \ker i$, and there is an $\tilde{a}\in \tilde{H}_{\Kh}^{-1,s-2}(K;\Z)$ with $\beta(\tilde{a}) = 2\tilde{b}$. By Lemma \ref{lm:firstdiff} $d_1(\tilde{a}) \pm \tilde{b}$ is an element of order $2$. As we assume $\tilde{H}_{\Kh}^{0,s}(K;\Z)$ to not have $2$-torsion, we have $d_1(\tilde{a})=\pm \tilde{b}$.

If we denote $\tilde{b}_2=p(\tilde{b})$, $b_2\in p(b)$, where $p$ is still change of coefficients from $\Z$ to $\F_2$, and we further write $d_{1,\F_2}\colon E^{-1,s-2}_{1,\F_2} \to E^{0,s}_{1,\F_2}$ for the differential in the spectral sequence, we get that $\tilde{b}_2$ is in the image of $d_{1,\F_2}$.

Since the unreduced spectral sequence with coefficients in $\F_2$ splits into two copies of the reduced version by Theorem \ref{thm:splitspec}, we also get that $b_2$ is in the image of $\bar{d}_{1,\F_2}\colon H^{-1,s-3}_{\Kh}(K;\F_2)\to H^{0,s-1}_{\Kh}(K;\F_2)$.

Then $b_2$ is send to $0$ in $H^0(\mathcal{F}_{s-3}/\mathcal{F}_{s+1};\F_2)$, the next term in the long exact sequence having $\bar{d}_{1,\F_2}$ as connecting homomorphism, and therefore it is also send to $0$ in $H^0(\CBN(K)/\mathcal{F}_{s+1};\F_2)$. Since every element of $B$ can be written as a $b_2$, we get that $s^{\Sq^1}_+$ is trivial for $K$ by \cite[Prop.6.3(2)]{MR4244204}.
\end{proof}

\begin{corollary}
Let $K$ be a $1$-thin knot. Then $s^{\Sq^1}_+$ is trivial for $K$.\hfill\qed
\end{corollary}

\section{Computations}

\subsection{Prime knots with up to 18 crossings}

We use the knot tables from \cite{MR4117738} to calculate $s^\Z(K)$ for every prime knot with up to $18$ crossings. The computations were done with the latest version of \verb+knotjob+, which is available from the author's website.

For all of these knots $K$ we get $\gl(K)\leq 1$, and there are only around 18,500 knots with $\gl(K) = 1$. Furthermore, up to $17$ crossings we get $\gl(K) = 1$ is equivalent to $s^\Z(K) = (s^\Q(K),2)$ and $s^{\F_2}(K) \not= s^\Q(K)$.

For knots with $18$ crossings there are a few exceptions. Firstly, there are two knots, $18^{nh}_{5566876}$ and the mirror of $18^{nh}_{37144251}$, for which $s^\Z(K) = (s^\Q(K),3)$. In both of these cases $s^{\F_3}(K) = 0$, while $s^\Q(K)\not=0$. See Figure \ref{fig:knotf3} for a diagram of $18^{nh}_{5566876}$.

\begin{figure}[ht]
\includegraphics[width = 5cm]{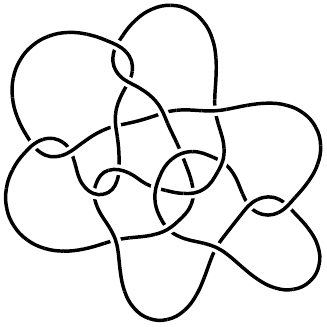}
\caption{\label{fig:knotf3}The knot $K=18^{nh}_{5566876}$, one among only two knots with at most $18$ crossings and $s^{\F_3}(K)\not=s^\Q(K)$.} 
\end{figure}

Secondly, there are $25$ knots $K$ for which $s^{\Z}(K) = (s^\Q(K), 2)$, yet $s^{\F_2}(K) = s^{\Q}(K)$. The knot $18^{nh}_{9772775}$ is the one with the smallest hyperbolic volume among them. Interestingly, in all of these cases we get $s^\Q(K) = 2$, so that we do not get an improvement on the slice genus from Theorem \ref{thm:genusbound}. On the other hand, connected sum of such a knot with an appropriate trefoil knot gives an example of a knot with $\gl(K) = 1$ and $s^\F(K)=0$ for every field $\F$.

\begin{proof}[Proof of Theorem \ref{thm:firstmain}]
The knot $K=18^{nh}_{23345595}$ is another one of the $25$ knots with $s^\Q(K)=s^{\F_p}(K)$ for all $p$, and $s^\Z(\overline{K}) = (2,2)$. Hence $s^\Q(K\# T(2,3)) = s^{\F_p}(K\# T(2,3)) = 0$ for all $p$, and a computation shows that $s^\Z(\overline{K\# T(2,3)}) = (0,2)$. Another calculation shows that the signature of this knot is also $0$, and the determinant is equal to $45^2$.
\end{proof}

Connected sums give convenient ways to produce knots with $\gl(K) > 1$. For example, $s^\Z(\overline{14^{ns}_1} \# \overline{14^{ns}_1}) = (0, 2, 2)$. Here we use the notation of \cite{MR4117738}, and $14^{ns}_1$ is a Whitehead double of the trefoil. In \cite[\S 4]{iltgen2021khovanov} it is shown that $14^{ns}_1$ contains a {\em staircase complex} $S_1$. This complex is in fact the dual of the complex in Example \ref{ex:staircase}. There exist staircase complexes $S_n$ for $n\geq 1$ and it is shown in \cite[Lm.4.3]{iltgen2021khovanov} that $S_{n+1}$ contains $S_1\otimes S_n$ as a direct summand. From the particular form of the staircase complex $S_n$ it now follows easily that
\begin{align*}
\gl\left((\overline{14^{ns}_1})^{\# n}\right) &= n \\
s^\Z\left((\overline{14^{ns}_1})^{\# n}\right) &=  (0,2,\ldots,2)
\end{align*}
for $n\geq 2$.

From the K\"unneth formula we get that
\[
\gl(K_1\# K_2) \leq \gl(K_1)+\gl(K_2)
\]
for any knots $K_1, K_2$. Since $\gl(K \# \overline{K}) = 0$ by Theorem \ref{thm:sliceinv} we cannot get equality here. For a more unusual example we calculated that $\gl(\overline{14^{ns}_1}\# \overline{18^{nh}_{9772775}}) = 1$, and, in fact, $s^\Z(\overline{14^{ns}_1}\# \overline{18^{nh}_{9772775}}) = (2, 4)$.
\subsection{Whitehead doubles}

Knots for which $\gl(K)>0$, and hence where the $s$-invariant can differ over different fields, appear to be rather rare from the above calculations. On the other hand, among the $380$ satellite knots with up to $19$ crossings listed in \cite{MR4117738} there are $85$ which have $\gl(K) = 1$ (and all of these satisfy $s^{\F_2}(K) \not=s^{\Q}(K)$). 

The one with the least number of crossings is a Whitehead double of the trefoil, also known as $14^n_{22180}$ in knotscape notation. Furthermore, the first discovered example of a knot with $s^{\F_3}(K) \not=s^{\Q}(K)$ is also a Whitehead double (of the $(3,4)$-torus knot), see \cite{LewarkZib}.

It seems therefore natural to investigate Whitehead doubles in general. Calculation times can vary greatly, already for companion knots with small number of crossings, so our investigation is not as thoroughly as we would have liked. However, we get a few more cases where $s^\Z(K)$ differs from $s^\Q(K)$ and the cyclic group has more than $3$ elements. We list these examples in the following table.
\begin{table}[h]
\begin{tabular}{c|c|c}
Knot $K$ & $s^{\Z}(K)$  & $s^{\Z}(\overline{K})$ \\
\hline
\hline
$W_+(5_1, 6)$ & $0$ & $(0,4)$ \\
\hline
$W_+(7_3, 6)$ & $0$ & $(0,4)$ \\
\hline
$W_+(7_5, 6)$ & $0$ & $(0,4)$ \\
\hline
$W_+(9_3, 9)$ & $0$ & $(0,8)$ \\
\hline
$W_+(10_{124}, 11)$ & $0$ & $(0,6)$ \\
\hline
$W_+(T(2,7), 9)$ & $0$ & $(0,8)$ \\
\hline
$W_+(T(2,9), 12)$ & $0$ & $(0,16)$ \\
\hline
$W_+(T(2,47), 69)$ & $0$ & $(0, 2^{23})$ 
\end{tabular}
\caption{\label{tab:one}The integral $s$-invariant of Whitehead doubles of some knots from Rolfsen's table.}
\end{table}
We should point out that we only sporadically checked knots with at least $8$ crossings, so there may be more interesting examples coming from the Rolfsen table. Notice that $10_{124}$ is the torus knot $T(3,5)$, so Whitehead doubles of torus knots appear to be particularly interesting. They also appear computationally more feasible.

The last three examples suggest the following conjecture.

\begin{conjecture}\label{con:whitetorus}
Let $K$ be the mirror of $W_+(T(2,2n+1), 3n)$ with $n$ a positive integer. Then
\[
s^\Z(K) = (0, 2^n).
\]
\end{conjecture}

Conjecture \ref{con:whitetorus} appears to be closely related to \cite[Conj.3.14]{MR3894728}, where it is conjectured that the flat $2$-cabling of $T(2,2n+1)$ contains torsion of order $2^n$ in its Khovanov homology. While this conjecture is formulated for unreduced Khovanov homology, computations show one can extend it to reduced Khovanov homology.

The reduced Khovanov homology of the Whitehead double fits in a long exact sequence with two of the links from \cite{MR3894728}, so it is unsurprising that it also contains torsion of large powers of $2$ (although this is not automatic). Adding twists in the Whitehead double is going to preserve this torsion, but shifts it into different bi-degrees, compare \cite[Thm.3.1]{MR3894728}. Computations for low values of $n$ then show that the Whitehead doubles described in Conjecture \ref{con:whitetorus} have this torsion in homological degree $0$. One still needs to check this torsion survives to the $E_\infty$ page of the spectral sequence.

Returning to the observation of \cite{LewarkZib} that $K=W_+(T(3,4),8)$ satisfies $s^\Q(K)\not=s^{\F_3}(K)$, together with the analogous statement holding for $W_+(T(2,3),3)$, we now arrive at the following conjecture, which has also been conjectured by Lewark and Zibrowius.

\begin{conjecture}
Let $p$ be a prime. For $K = W_+(T(p,p+1), p^2-1)$ and $q$ a prime different from $p$ we have
\[
s^\Q(K) = s^{\F_q}(K) \not= s^{\F_p}(K).
\]
\end{conjecture}

This conjecture would prove Conjecture \ref{con:bigconj}.

In terms of $s^\Z(K)$ we can also look at non-prime integers, and state the following conjecture.

\begin{conjecture}
Let $n\geq 2$. For $K = W_+(T(n,n+1), n^2-1)$ we have
\[
s^\Z(\overline{K}) = (0,n).
\]
\end{conjecture}

Computationally this has been checked for $n< 4$, for $n=4$ calculations of $s^\F$ where $\F$ has small characteristic, $s^{\Sq^1}$ and $\tilde{H}_{\Kh}$ suggest $(0,4)$. Ffor $n>4$ this seems out of reach for computations.

\subsection{The $\Sq^1$-refinement of $s^{\F_2}$}

In a previous paper \cite{MR4244204} we calculated the $\Sq^1$-refinements for all knots up to $16$ crossings, and also a few larger knots.  For all of our calculations the following two statements about a knot $K$ were equivalent.
\renewcommand{\theenumi}{S.\arabic{enumi}}
\begin{enumerate}
\item \label{enu:one}We have $s^{\F_2}(K) \not= s^\Q(K)$.
\item We have $r^{\Sq^1}_+(K)$ or $r^{\Sq^1}_-(K)$ non-trivial.\label{enu:two}
\end{enumerate}
Furthermore, we also always observed
\begin{enumerate}
\setcounter{enumi}{2}
\item We have $r^{\Sq^1}_+(K)$ non-trivial if and only if $s^{\Sq^1}_+(K)$ non-trivial.\label{enu:three}
\item If $r^{\Sq^1}_+(K)$ is non-trivial, then $r^{\Sq^1}_-(K)$ is trivial.\label{enu:four}
\end{enumerate}

There do not seem to be algebraic reasons why this should be true, and, in fact, we now have examples showing that (\ref{enu:one}) and (\ref{enu:two}) are independent of each other.

For any of the $25$ knots with $s^{\Z}(K) = (s^\Q(K), 2)$, yet $s^{\F_2}(K) = s^{\Q}(K)$ we have (\ref{enu:two}) is true. Also, for all the knots in Table \ref{tab:one} with last column $(0,4)$ we have (\ref{enu:two}) is false, while (\ref{enu:one}) is true.

There seems to be a relation between having torsion of order $2^k$ with $k>1$ in $s^\Z$, and $r^{\Sq^1}$ being trivial. After all, the first Steenrod square does not detect such torsion. On the other hand, the $\Sq^1$-refinements look at the $E_1$ page of the spectral sequence, while $s^\Z$ looks at the $E_\infty$ page, and algebraically there does not seem to be a direct link.

We still do not have counterexamples for (\ref{enu:three}) and (\ref{enu:four}) after checking knots with up to $17$ crossings. 

\bibliography{KnotHomology}
\bibliographystyle{amsalpha}

\end{document}